\documentclass[12pt]{article}

\usepackage[utf8x]{inputenc}
\usepackage{amsmath,ifthen,color}

\usepackage{latexsym,amssymb,url}
\usepackage{graphicx}
\usepackage{enumerate}
\usepackage{cite}

\usepackage{hyperref}

\hypersetup{
  pdftitle = {Minors in graphs of large pumpkin-girth},
  pdfauthor = {D. Chatzidimitriou, J-.F. Raymond, I. Sau and D. M. Thilikos},
   colorlinks = true,
  linkcolor = black!30!red,
  citecolor = black!30!green
}
\newcommand{\itemref}[1]{\hyperref[#1]{(\ref*{#1})}}

\usepackage{tikz}
\tikzset{black node/.style={draw, circle, fill = black, minimum size = 7pt, inner sep = 0pt}}
\tikzset{diam/.style={draw, diamond, fill = black, minimum size = 7pt, inner sep = 0pt}}
\tikzset{white node/.style={draw, circle, fill = white, minimum size = 7pt, inner sep = 0pt}}
\tikzset{normal/.style={draw=none, rectangle, fill = none, minimum size = 0pt, inner sep = 0pt}}

\usepackage{amsthm}
\usepackage[capitalize]{cleveref} 

\def\claimqed {{
\parfillskip=0pt        
\widowpenalty=10000     
\displaywidowpenalty=10000  
\finalhyphendemerits=0  
\leavevmode             
\unskip                 
\nobreak                
\hfil                   
\penalty50              
\hskip.2em              
\null                   
\hfill                  
$\Diamond$

\par
\medskip
}}                  

\newtheorem{theorem}{Theorem}[section]
\newtheorem{proposition}[theorem]{Proposition}
\newtheorem{lemma}[theorem]{Lemma}
\newtheorem{corollary}[theorem]{Corollary}
\newtheorem{observation}[theorem]{Observation}
\newtheorem{claim}[theorem]{Claim}
\newenvironment{proofclaim}[2]{\par\noindent\textit{Proof~#1.}\space#2}{{\claimqed}}

\newcommand{\intv}[2]{\left \{ #1, \dots, #2\right \}}

\DeclareMathOperator{\tw}{{\bf tw}}
\DeclareMathOperator{\bw}{{\bf bw}}
\DeclareMathOperator{\polylog}{{polylog}}

\newcommand{\N}{\mathbb{N}}
\newcommand{\R}{\mathbb{R}}

\newcommand{\eaddress}[1]{\href{mailto:#1}{\texttt{#1}}}

\title{Minors in graphs of large $\theta_r$-girth\thanks{The second author has been partially supported by the
    Warsaw Centre of Mathematics and Computer Science and by the
    (Polish) National Science Centre grant PRELUDIUM
    2013/11/N/ST6/02706. The first and the last authors 
    were co-financed by the E.U. (European Social Fund - ESF) and Greek national funds through the Operational Program ``Education and Lifelong Learning'' of the National Strategic Reference Framework (NSRF) - Research Funding Program: ``Thales. Investing in knowledge society through the European Social Fund''.
    Email addresses:
    \eaddress{hatzisdimitris@gmail.com},
    \eaddress{jean-florent.raymond@mimuw.edu.pl},
    \eaddress{ignasi.sau@lirmm.fr}, and
    \eaddress{sedthilk@thilikos.info}.}}
\author{Dimitris Chatzidimitriou\thanks{Department of Mathematics, National and Kapodistrian  University of Athens, Athens, Greece.} \and Jean-Florent Raymond\thanks{AlGCo project team, CNRS, LIRMM, France.}~\thanks{University of Montpellier, Montpellier, France.}~\thanks{Faculty of Mathematics, Informatics and Mechanics, University of Warsaw, Warsaw, Poland.} \and  Ignasi Sau$^\ddag$ \and Dimitrios M. Thilikos$^\dag$$^\ddag$}

\date{\empty}

\begin{document}
\maketitle

\begin{abstract}\noindent For every $r \in \N$, let $\theta_r$ denote the graph with two vertices and $r$ parallel edges. The \emph{$\theta_r$-girth} of a graph $G$ is the minimum number of edges of a subgraph of $G$ that can be contracted to $\theta_r$. This notion generalizes the usual concept of girth which corresponds to the case~$r=2$.
In {[Minors in graphs of large girth, {\em Random Structures \& Algorithms}, 22(2):213--225, 2003]},  Kühn and  Osthus showed that graphs of sufficiently large minimum degree contain clique-minors whose order is an exponential function of their girth.
We extend this result for the case of $\theta_{r}$-girth 
and we show that the minimum degree can be replaced by some connectivity
measurement.
As an application of our results, we prove that, for every fixed $r$, graphs 
excluding as a minor the disjoint union of $k$ $\theta_{r}$'s have treewidth 
$O(k\cdot  \log k)$.
\smallskip
\end{abstract}



\section{Introduction}

A classic result in graph theory asserts that if a graph has minimum degree $ck\sqrt{\log k}$, then 
it can be transformed to a complete graph on at least $k$ 
vertices by applying edge contractions and vertex deletions (i.e., it contains a  {\em ${k}$-clique minor}).
 This result has been proven by Kostochka in~\cite{kosto84}
and  Thomason in~\cite{1983MPCPS} and a precise estimation of the 
constant  $c$ has been given by Thomason in~\cite{Thomason2001318}. For recent 
results related to conditions that force a clique minor see~\cite{Markstrom04,JoretW13,DujmovicHJRW13,FountoulakisKO09,KuhnO04b}.

The {\em girth} of a graph $G$ is the minimum length of a cycle in~$G$.
Interestingly, it follows 
that  graphs of large minimum degree contain clique-minors
whose order is an {\em exponential} function of their  girth.
In particular, it follows by the main result of Kühn and Osthus in~\cite{RSA:RSA10076}
 that  there is a  constant $c$ such that,
if a graph has minimum degree $d\geq 3$ and girth $z$, then 
it contains as a minor a clique of size $k$, where
\[
k\geq \frac{d^{cz}}{\sqrt{z\cdot \log d}}.
\]

In this paper we provide conditions, alternative to the above one, that can force the existence of a clique-minor whose size  is exponential.
\paragraph{$H$-girth.} We say that a graph $H$ is a {\em minor} of a graph $G$, if $H$ can be obtained from $G$ by using the operations of vertex removal, edge removal, and edge contraction. An \emph{$H$-model} in $G$ is
a subgraph of $G$ that contains $H$ as a minor. Given two graphs  $G$ and $H$, we define 
the $H$-{\em girth} of $G$ as the minimum number of edges of an $H$-model in $G$. 
If $G$ does not contain $H$ as a minor, we will say that its $H$-girth is equal to infinity.
For every $r \in \N$, let $\theta_r$ denote the graph with two vertices and $r$ parallel edges, e.g. in \cref{fig:t5} the graph $\theta_5$ with 5 parallel edges.
Clearly, the girth of a graph is its $\theta_{2}$-girth and,  for every $r_{1}\leq r_{2}$,  
the $\theta_{r_{1}}$-girth of a graph is at most its $\theta_{r_{2}}$-girth.
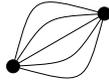
\begin{figure}[ht]
\label{pumkin}
 \centering
 
\scalebox{.7}{\begin{tikzpicture}[rotate=30, every node/.style = black node]
  \node (x) at (0,0) {};
  \node (y) at (2,0) {};
  \draw (x) .. controls (1, 1) and (1, 1) .. (y);
  \draw (x) .. controls (1, -1) and (1, -1) .. (y);
  \draw (x) .. controls (1, .5) and (1, .5) .. (y);
  \draw (x) .. controls (1, -.5) and (1, -.5) .. (y);
  \draw (x) -- (y);
 \end{tikzpicture}}
 \caption{The graph $\theta_5.$}
 \label{fig:t5}
\end{figure}

Our first result is the following extension of the  result of Kühn and Osthus in~\cite{RSA:RSA10076} for the case of $\theta_{r}$-girth.

\begin{theorem}\label{t:klikforce}
There is a constant $c$ such that, for every $r\geq 2$, $d\geq 3r$, and $z\geq 2r$,
if a graph has minimum degree $d$ and $\theta_{r}$-girth at least 
$z$, then it contains as a minor a clique of size $k$,
where  
\[
k\geq \frac{(\frac{d}{r})^{c \frac{z}{r}}}{\sqrt{\frac{z}{r}\cdot \log d}}.
\]
\end{theorem}
In the formula above, a lower bound on the minimum degree as a
function of $r$ is necessary. An easy computation shows that when
applying \cref{t:klikforce} for $r=2$, we can get the
aforementioned formula of Kühn and Osthus, where the constant in the
exponent is one fourth of
the constant of \cref{t:klikforce}.

Our second finding
is that this degree condition can be replaced by some ``loose connectivity'' requirement. 

\paragraph{Loose connectivity.}
For two integers $\alpha, \beta \in \N$, a graph $G$ is called {\em $(\alpha,\beta)$-{loosely connected}} if for every $A,B \subseteq V(G)$ such that $V(G) = A \cup B$ and $G$ has no edge between $A\setminus B$ and $B \setminus A$, we have that $|A \cap B| < \beta \Rightarrow \min(|A\setminus B|, |B\setminus A|) \leq  \alpha$. Intuitively, this means that a \emph{small} separator (i.e., on less than $\beta$ vertices) cannot ``split'' the graph into two \emph{large} parts (that is, with more than $\alpha$ vertices each). 

Our second result indicates that the requirement on the minimum degree in~\cref{t:klikforce}  can be replaced by the loose connectivity 
 condition  as follows.

\begin{theorem}\label{t:klikforce2}
There is a  constant $c>0$ such that, for every $r\geq 2$, $\alpha\geq 1$, and
$z\geq 168\cdot\alpha\cdot r\log r$, it holds 
that if a connected graph with at least $z$ edges
is $(\alpha,2r - 1)$-loosely connected and 
has $\theta_{r}$-girth at least 
$z$, then it contains as a minor a clique of size $k$, where  
\begin{align*}
k\geq \frac{2^{c\cdot\frac{z}{r\alpha}}}{\sqrt{rz}}.
\end{align*}
\end{theorem}

Both \cref{t:klikforce} and~\cref{t:klikforce2} are derived from two more general results, namely \cref{sl0tr} and \cref{rk6ter}, respectively. \cref{sl0tr} asserts that
graphs with large $\theta_{r}$-girth and sufficiently large minimum degree contain as a minor a graph whose 
minimum degree is exponential in the girth.
\cref{rk6ter} replaces the minimum degree condition with the absence 
of sufficiently large ``edge-protrusions'', that are roughly tree-like structured subgraphs with small boundary 
to the rest of the graph (see \cref{sec:def} for the detailed definitions).

\paragraph{Treewidth.}
A \emph{tree-decomposition} of a graph $G$ is a pair $(T,{\cal X})$
where $T$ is a tree and ${\cal X}$ is a family of subsets of $V(G)$, called \emph{bags},
 indexed by the vertices of $T$ and such that:
 \begin{enumerate}[(i)]
 \item for each edge $e=(x,y)\in E(G)$ there is a vertex $t\in V(T)$
   such that $\{x,y\} \subseteq X_t$;
 \item for each vertex $u \in V(G)$ the subgraph of $T$ induced
   by $\{ t\in V(T) \mid u\in X_t\}$ is connected; and
 \item $\bigcup_{t \in V(T)} X_t =V(G)$. 
 \end{enumerate}

The \emph{width} of a tree-decomposition $(T,{\cal X})$ is the maximum
size of its bags minus~one. The \emph{treewidth}
of a graph $G$, denoted $\tw(G)$, is defined as the minimum width over
all tree-decompositions of $G$. 

Treewidth has been introduced in the Graph Minors Series 
of Robertson and Seymour~\cite{RobertsonS86GMV}
and is an important parameter in both combinatorics 
and algorithms.  In~\cite{RobertsonS86GMV}, Robertson and Seymour
proved that for every planar graph $H$, there exists a constant $c_{H}$
such that every graph excluding $H$ as a minor has treewidth at most $c_{H}$.
This result has several applications in algorithms and a lot of research has been devoted 
to optimizing the constant $c_{H}$ in general or for specific instanciations of $H$ (see~\cite{RST94,DiestelJGT99}).
In this direction, Chekuri and Chuzhoy proved in~\cite{ChekuriC13poly,Chuzhoy15grid}  
that $c_{H}$ is bounded by  a polynomial  in the size 
of $H$. Specific results for particular $H$'s such that $c_{H}$ is a low polynomial 
function have been derived in~\cite{Bodl93a,BodlaenderLTT97,Birmele06,RaymondT13lowp}.

Given a graph $J$, we denote by $k\cdot J$ the disjoint union of $k$ copies of $J$.
A consequence of the general results of Chekuri and Chuzhoy in~\cite{ChekuriC13larg} is that for every planar graph~$J$, it holds that~$c_{k\cdot J}= k\cdot \polylog k$, where $\polylog k$ denotes some polynomial in $\log k$. Prior to this,  
a quadratic (in $k$) upper bound was derived for the case where $J=\theta_{r}$~\cite{FominLMPS13, Birmele06}.
As an application of our results, we prove that for every 
fixed $r$,  $c_{k\cdot \theta_{r}}=O(k\cdot \log k)$ (\cref{s:new}). We also argue that this 
bound is {\sl tight} in the sense that it cannot be improved to $o(k\cdot \log k)$. Our proof 
is based on \cref{rk6ter} and the results of Geelen, Gerards, Robertson, and Whittle
on the excluded minors for the matroids of branch-width $k$~\cite{GeelenGR03onth}.

%

%



\paragraph{Organization of the paper.}
The main notions used in this paper are defined in \cref{sec:def}. Then, we show in \cref{sec:s1} that the proofs of \cref{t:klikforce} and \cref{t:klikforce2} can be derived from \cref{sl0tr} and \cref{rk6ter}, which are proved in~\cref{u84p1gm6}. Finally, in \cref{sec:last}, we prove 
our tight bound on the minor-exclusion of $k\cdot \theta_{r}$.

\section{Definitions}
\label{sec:def}

Given a function $\phi: A\rightarrow B$ and a set $C\subseteq A$, we
define $\phi(C)=\{\phi(x)\mid x\in C\}$. Let $\chi,\psi: \N
\rightarrow \N$.
We say that $\chi(n)=O_{r}(\psi(n))$ if there exists a function $\phi:\N
\rightarrow \N$ such that, for every $r\in \N$, $\chi(n)=O(
\phi(r)\cdot \psi(n))$. This notation indicates that the contribution
of $r$ is hidden in the constant of the big-O notation.
If $\mathcal{X}$ is a set of sets, we denote by $\bigcup \mathcal{X}$
the union $\bigcup_{X \in \mathcal{X}} X$.
Unless otherwise specified, logarithms are binary.

\paragraph{Graphs.} All graphs in this paper are finite, undirected, loopless,
and may have multiple edges. For this reason, a graph is represented
by a pair $G=(V,E)$ where $V$ is its vertex set, denoted by $V(G)$ and
$E$ is its edge multi-set, denoted by $E(G)$.
In this paper, when giving
the running time of an algorithm involving some graph $G$, we agree
that  $n=|V(G)|$ and $m=|E(G)|$.
Given a vertex $v$ of a graph $G$, the set of vertices of $G$ that are
adjacent to $v$ is denoted by $N_G(v)$ and the {\em degree} of $v$ in $G$ is $|N_G(v)|$. Observe that since multiple edges are allowed, the degree of a vertex may differ from the number of incident edges.
For every subset $S \subseteq
V(G)$, we set $N_G(S) = \bigcup_{v\in S}N_G(v) \setminus S$ (all
vertices of $V(G) \setminus S$ that have a neighbor in $S$). The
minimum degree over all vertices of a graph $G$ is denoted by~$\delta(G)$.
 For a given graph $G$ and two vertices $u, v \in V(G)$, ${\bf
dist}_{G} (u, v)$ denotes {\em the distance between $u$ and $v$},
which is the number of edges on a shortest path between $u$ and $v$,
and ${\bf diam}(G)$ denotes $\max \{ {\bf dist}_{G}(u,v)\mid {u, v\in
V(G)} \}$.  For a set $S \subseteq V(G)$ and a vertex $w \in V$, ${\bf
dist}_{G}(S,w)$ denotes $\min\{{\bf dist}_{G}(v,w)\mid {v\in S} \}$.
Also, for a given vertex $u\in V(G)$, ${\bf ecc}_{G}(u)$ denotes the
{\em eccentricity} of the vertex $u$, that is, $\max \{ {\bf dist}_{G}(u,v)
\mid {v\in V(G)} \}$.

\paragraph{Rooted trees.}
A {\em rooted tree} is a pair $(T,s)$ such that $T$ is a tree and $s$, which we call
the \emph{root}, belongs to $V(T)$. Given a vertex $x\in V(T)$, the
{\em descendants} of $x$ in $(T,s)$ are the elements of ${\bf des}_{(T,s)}(x)$, which is defined as the set containing each vertex
$w$ such that the unique path from $w$ to $s$ in $T$ contains $x$.  Given a
rooted tree $(T,s)$ and a vertex $x\in V(G)$, the {\em height} of $x$
in $(T,s)$ is the maximum distance between $x$ and a vertex in ${\bf
des}_{(T,s)}(x)$.  The {\em height} of $(T,s)$ is the height of $s$ in
$(T,s)$.  The {\em children} of a vertex $x\in V(T)$ are the vertices
in ${\bf des}_{(T,s)}(x)$ that are adjacent to $x$.  A {\em leaf} of
$(T,s)$ is a vertex of $T$ without children. Notice that, according to this definition, $s$ is not a leaf unless $|V(T)|=1$.
The {\em parent} of a
vertex $x\in V(T)\setminus \{s\}$, denoted by ${\bf p}(x)$, is the
unique vertex of $T$ that has $x$ as a child.

\paragraph{Partitions and protrusions.} A {\em rooted
  tree-partition} of a graph $G$ is a triple ${\cal D}=({\cal X},T,s)$
where
$(T,s)$ is a rooted tree and ${\cal X}=\{X_t\}_{t\in V(T)}$ is a
partition of $V(G)$ where either $|V(T)|=1$ or for every $\{x,y\}\in
E(G)$, there exists an edge $\{t,t'\}\in E(T)$ such that
$\{x,y\}\subseteq X_{t}\cup X_{t'}$ (see also
\cite{Seese,Halin1991203,Ding199645}). The elements of ${\cal X}$ are
called~\emph{bags}. In other words,
the endpoints of every edge of $G$ either belong to the same bag, or they
belong to bags of adjacent vertices of $T$.  Given an edge $f=\{t,t'\}\in
E(T)$, we define $E_{f}$ as the set of edges with one endpoint in
$X_{t}$ and the other in $X_{t'}$. 
The {\em width} of ${\cal D}$ is defined as
$\max\{|X_{t}|\}_{t\in V(T)}\cup\{|E_{f}|\}_{f\in E(T)}$.

In order to decompose graphs along edge cuts, we introduce the
following edge-counterpart of the notion of
\emph{(vertex-)protrusion} used
in~\cite{BodlaenderFLPST09,BodlaenderFLPST09arxiv} (among others).
A subset $Y\subseteq V(G)$ is a
{\em $t$-edge-protrusion} of $G$ with {\em extension} $w$ (for some
positive integer $w$) if the graph $G[Y\cup N_{G}(Y)]$ has a rooted
tree-partition ${\cal D}=({\cal X},T,s)$ of width at most $t$ and such
that $N_{G}(Y)=X_{s}$ and $|V(T)|\geq w$. The protrusion $Y$ is
said to be \emph{connected} whenever $Y\cup N_{G}(Y)$ induces a
connected subgraph in~$G$.

\paragraph{Distance-decompositions.}
  A \emph{distance-decomposition} of a connected graph
  $G$ is a rooted tree-partition
  ${\cal D}=({\cal X}, T, s)$ of $G$, where the following additional
  requirements are~met (see also~\cite{Distwidth}):
 \begin{enumerate}[(i)]
 \item $X_s$ contains only one vertex, we shall call it  $u$, referred to as the \emph{origin} of $\mathcal{D}$;
 \item for every $t\in V(T)$ and every $x\in X_{t}$,  ${\bf dist}_{G}
   (x, u)={\bf dist}_{T} (t,s)$;\label{it:dist}
 \item for every $t\in V(T)$, the graph $G_{t}=G\left [\bigcup_{t'\in{\bf
       des}_{(T,s)}(t)}X_{t'} \right ]$ is connected; and \label{it:conn}
 \item if $C$ is the set of children of a vertex $t\in V(T)$, then the
   graphs $\left \{G_{t'} \right\}_{t' \in C}$ are the connected
   components of~$G_t\setminus X_{t}$.
 \end{enumerate}
An example of distance-decomposition is given in~\cref{fig:dd}. For every vertex $u$ of a graph on $m$ edges, a distance-decomposition
$({\cal X}, T, s)$ with origin $u$ can be constructed in
$O(m)$ steps by~breadth-first search.

\begin{figure}[ht]
  \centering
\scalebox{.9}{\begin{tikzpicture}[every node/.style = black node,scale=1]
    \begin{scope}[rotate = 90]
      \draw (2,0) node[label=-90:$u_5$] (u) {} -- (1,1) node[label=180:$u_6$] {} -- (0,0) node[label=-90:$u_8$] {} -- (1,-1) node[label=0:$u_7$] {} -- (u) -- (4,0) node[label=0:$u_3$] (v) {} -- (3,1) node[label=180:$u_4$] {} -- (u) (5,-1) node[label=90:$u_1$] {} -- (v) -- (5,1) node[label=90:$u_0$] {} -- (4,1) node[label=180:$u_2$] {} -- (v);
    \end{scope}
    \begin{scope}[xshift = 5cm, yshift = 4cm, scale = 1.5, every node/.style = {draw=none, color = black, fill = white, rectangle, rounded corners}]
      \draw (0,0) node (n2) {$\{u_5\}$}
      (-1,-1) node (n01) {$\{u_6, u_7\}$}
      (1,-1) node (n34) {$\{u_3, u_4\}$}
      (-1,-2) node (n-1) {$\{u_8\}$}
      (0.5,-2) node (n5) {$\{u_0, u_2\}$}
      (1.5,-2) node (n6) {$\{u_1\}$};
      \draw (n-1) -- (n01) -- (n2) -- (n34) -- (n5) (n34) -- (n6);
    \end{scope}
  \end{tikzpicture}
  }
  \caption{A graph (left) and a distance-decomposition with origin $u_5$ of it (right).}
  \label{fig:dd}
\end{figure}
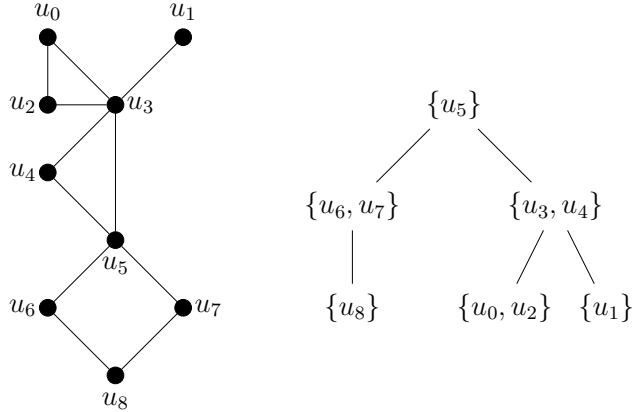

For every $t\in V(T)\setminus\{s\}$, we define $E^{(t)}$ as the set of edges that have
one endpoint in $X_{t}$ and the other in $X_{{\bf p}(t)}$. 

Let $P$ be a path in $G$ and ${\cal D}=({\cal X}, T, s)$ a distance-decomposition of $P$.
We say that $P$ is a {\em
  straight path} if the heights, in $(T,s),$ of the indices of the
bags in ${\cal D}$ that contain vertices of $P$  are pairwise
distinct. Obviously, in that case, the sequence of the heights of the
bags that contain each subsequent vertex of the path is strictly
monotone.

\paragraph{Grouped partitions.}\label{par:group}
 Let $G$ be a connected graph and let $d\in \N$. A {\em $d$-grouped
   partition} of $G$ is a partition ${\cal R}=\{ R_1, \ldots,
 R_l\}$ of $V(G)$ (for some positive integer $l$) such that for
 each $i\in \left \{1, \dots, l \right \}$, the graph $G[R_{i}]$ is connected  
 and  there is a vertex $s_i \in R_i$ with the following properties:
 \begin{enumerate}[(i)]
 \item ${\bf ecc}_{G[R_i]}(s_i) \leq2d$ and\label{it:gp1}
 \item for each edge $e=\{x,y\}\in E(G)$ where $x\in R_i$ and $y\in R_j$
   for some distinct integers $i,j \in \left\{ 1, \dots, l\right \}$,
   it holds that ${\bf dist}_{G}(x, s_i)\geq d$ and ${\bf dist}_{G}(y,
   s_j)\geq d$. \label{it:gp2}
 \end{enumerate}
A set $S=\{s_{1},\ldots,s_{l}\}$ as above is  a  {\em set of centers}
of  ${\cal R}$ where $s_{i}$ is the {\em center} of $R_{i}$ for $i\in
\left \{1, \dots, l \right \}$.

Given a graph $G$, we define a {\em $d$-scattered set}  $W$ of $G$ as follows:
\begin{itemize}
\item $W\subseteq V(G)$ and
\item $\forall u,v \in W,\, {\bf dist}_{G}(u,v)>d$.
\end{itemize}
If $W$ is inclusion-maximal, it will be called a {\em maximal $d$-scattered
  set} of $G$.

\paragraph{Frontiers and ports.}\label{par:ports}
Let $G$ be a graph, let ${\cal R}=\{ R_1, \ldots, R_l\}$ be a $d$-grouped
partition of $G$, and let $S=\{s_{1},\ldots,s_{l}\}$ be a set
of centers of ${\cal R}$.
For every $i\in \intv{1}{l}$, we denote by  ${\cal D}_i=({\cal X}_i, T_{i},
s_i)$ the unique distance-decomposition with origin $s_i$ of the graph
$G[R_i]$ where ${\cal X}_{i}=\{X_{t}^{i}\}_{t\in V(T_{i})}$. For
every ${i}\in\intv{1}{l}$ and every $h\in \intv{0}{{\bf
    ecc}_{T_{i}}(s_i)}$,  we denote by $I_{i}^{h}$ the
vertices of $(T_{i}, s_i)$ that are at distance $h$ from
$s_i$, and we set $I_{i}^{< h}=\bigcup_{h' = 0}^{h-1}I_{i}^{h'}$ and
$I_{i}^{\geq  h}=\bigcup_{h'= h}^{{\bf
    ecc}_{T_{i}}(s_i)}I_{i}^{h'}$.
We also set 
\[
 V_{i}^{h}=\bigcup_{t\in I_{i}^{h}}X_{t}^{i},\qquad
 V_{i}^{<h}=\bigcup_{t\in I_{i}^{<h}}X_{t}^{i},\ \text{and}\qquad 
 V_{i}^{\geq h}=\bigcup_{t\in I_{i}^{\geq h}}X_{t}^{i}.
\]

The {\em vertex-frontier} $F_i$ of $R_{i}$ is the set of vertices
in $V_{i}^{d-1}$ that are connected in $G$ to a vertex $x \in V(G)\setminus
R_i$ via a path, the internal vertices of which belong to $V_{i}^{ 
\geq d}$. The {\em node-frontier} of $T_{i}$ is 
\begin{align}
N_{i} &= \{t\in V(T_{i}) \mid F_{i}\cap X_{t}\neq\emptyset\}\label{kfhdn48f}.
\end{align}
A vertex $t\in I_{i}^{\geq d-1}$ is called a {\em port} of $T_{i}$ if $X^i_{t}$
contains some vertex that is adjacent in $G$ to a vertex of $V(G)\setminus R_{i}$.

\section{\texorpdfstring{Finding small $\theta_r$-models}{Finding small models of a pumpkin}}
\label{sec:s1}

\subsection{Two intermediate results}
\label{finterm}

The main results of this section are the following.

\begin{theorem}
\label{sl0tr}
There exists an algorithm that, with input three integers $r,\delta,z$,
where $r\geq 2$, $\delta \geq 3r$, and $z\geq r$ and 
an $m$-edge graph $G$, outputs one the following:
\begin{itemize}
\item a $\theta_{r}$-model in $G$ with at most $z$ edges,
\item a vertex $v$ of $G$ of degree less than $\delta$, or 
\item an $H$-model in $G$ for some graph $H$ where $\delta(H)\geq \frac{\delta-2r+3}{r-1}\cdot {\lfloor \frac{\delta}{r-1}-1 \rfloor }^{\frac{z-r}{4r}},$
\end{itemize}
in $O_{r}(m)$ steps.
\end{theorem}

\begin{theorem}
\label{rk6ter}
There exists an algorithm that, with input three positive integers $r \geq 2, w, z$ and a connected $m$-edge graph $G$, where $m \geq z> r\geq 2$, outputs one of the following:
\begin{itemize} 
\item a $\theta_{r}$-model in $G$ with at most $z$ edges, 
\item a connected $(2r-2)$-edge-protrusion $Y$ of $G$ with extension
  more than $w$, or
\item an $H$-model in $G$ for some graph $H$ where   $\delta(H)\geq \frac{1}{r-1}2^{\frac{z-5r}{4r(2w+1)}}$,
\end{itemize}
in $O_{r}(m)$ steps.
\end{theorem}

The results of Chandran
and Subramanian in~\cite{Chandran200523} imply  that if $G$ has girth at least $z$ and 
minimum degree at least $\delta$, then $\tw(G)\geq \delta^{c\cdot z}$, for some constant $c$. 
As in the third condition of \cref{sl0tr} it holds that $\tw(G)\geq\tw(H)\geq \delta(H)$, \cref{sl0tr}
can also be seen as a qualitative 
extension of the results of~\cite{Chandran200523}. 
 
The above two results will be used to prove \cref{t:klikforce} and
\cref{t:klikforce2}. We will also need the following result of~Kostochka\cite{kosto84}.

\begin{proposition}[\!\! \cite{kosto84}, see also \cite{1983MPCPS,Thomason2001318}]\label{kosto}
There exists a  constant $\xi\in \R$ such that for every $d\in
\N$, every graph of average degree at least $d$ contains a clique of
order $k$ as a minor, for some integer $k$ satisfying
\[
k \geq \xi\cdot\frac{d}{\sqrt{\log d}}.
\]
\end{proposition}

\subsection{The proofs of the main theorems}
\label{sec:potmt}
We are now ready to prove \cref{t:klikforce} and \cref{t:klikforce2} using the intermediate results described in the previous section.

\begin{proof}[Proof of \cref{t:klikforce}.]
Observe that since $G$ has no $\theta_{r}$-model with at most $z$ edges
and $G$ has minimum degree
$d \geq 3r$, a call to the algorithm of \cref{sl0tr} on $(r,d,z,G)$ should
return an $H$-model of $G$, for some graph $H$ where
$\delta(H)\geq \frac{d-2r+3}{r-1}\cdot {\lfloor \frac{d}{r-1}-1
  \rfloor }^{\frac{z-r}{4r}}=:d'$.
Using the fact that $z-r\geq z/2$, 
  it is not hard to check that there is a constant $c' \in \R$ such that
\[
\xi\cdot\frac{d'}{\sqrt{\log d'}} \geq \frac{(\frac{d}{r})^{c'\cdot \frac{z}{r}}}{\sqrt{\frac{z}{r}\cdot \log d}}.
\]
Hence by \cref{kosto}, $G$ has a clique of the desired order as a~minor.
\end{proof}

\begin{proof}[Proof of~\cref{t:klikforce2}]
  As in the proof of \cref{t:klikforce}, the properties that
  $G$ enjoys will force a minor of large minimum degree. Let us call
  the algorithm of \cref{rk6ter} on $(r, 3\alpha, z, 
  G)$. We assumed that $G$ has no $\theta_r$-model with $z$
  edges or less, hence the output of the algorithm cannot be such a
  model.
Let us now assume that the algorithm outputs a
$(2r-2)$-edge-protrusion $Y$ with extension more than $3\alpha$, and let
$(\mathcal{X}, T,s)$ be a rooted tree-partition of $G[Y\cup N_{G}(Y)]$ of width at
most $2r-2$ such that $N_G(Y) = X_s$ and $n(T)> 3\alpha$.
It is known that every tree of order $n$ has a vertex, the removal of
which partitions the tree into components of size at most $n/2$ each.
Hence, there is a vertex $v \in V(T)$ and a partition $(Z,Z')$ of
$V(T) \setminus \{v\}$ such that:
\begin{itemize}
\item both $Z \cup \{v\}$ and $Z' \cup
\{v\}$ induce connected subtrees of $T$;
\item $\frac{1}{3}n(T) \leq |Z|, |Z'| \leq \frac{2}{3}n(T)$; and
\item $s \in Z$ or $v = s$.
\end{itemize}
Let $A = \bigcup_{t\in Z'}X_{t} \cup \{X_v\}$ and $B = V(G)
\setminus \bigcup_{t\in Z'}X_{t}$. Notice that $V(G) = A \cup B$ and that no edge of $G$ lies
between $A\setminus B$ and $B\setminus A$. As $A \cap B = X_v$, we have $|A \cap B| < 2r-1$.
Last, $Z' \subseteq A\setminus B$ and $Z \subseteq B\setminus A$ give that $|A\setminus B|,|B\setminus A| > \alpha$.
The existence of $A$ and $B$ contradicts the fact that $G$ is $(\alpha,
2r-1)$-loosely connected. Thus $G$ has no $(2r-2)$-edge-protrusion $Y$
of extension more than~$3\alpha$.

A consequence of this observation is that the only possible output of
the algorithm mentioned above is an $H$-model of $G$ for some
graph $H$, where
\begin{align*}
\delta(H)\geq
\frac{1}{r-1}\cdot2^{\frac{z-5r}{4r(6\alpha+1)}}\geq \frac{1}{r}\cdot 2^{\frac{z}{168\cdot r\alpha}}=:d.  
\end{align*}

Notice also that $\log d=\frac{z}{168\cdot r\alpha}-\log r$ which, by the condition of the 
theorem, is a non-negative number. Moreover, $\log d \leq z/r$.
Therefore, there is a constant $c'' \in \R$ such that
\[
\xi\cdot\frac{d}{\sqrt{\log d}} \geq \frac{ 2^{c''\cdot\frac{z}{ r\alpha}}}{\sqrt{z\cdot r}}
\]
in order to conclude the proof.
\end{proof}

\section{The proofs of \texorpdfstring{\cref{sl0tr} and \cref{rk6ter}}{the intermediate theorems}}
\label{u84p1gm6}

\subsection{Preliminary results}
Before proving \cref{rk6ter} and \cref{sl0tr} (in
\cref{sec:proofth5} and~\cref{sec:proofth4}, respectively) we
need some preliminary results. Let us start with some definitions.

Let $(T,s)$ be a rooted tree and let $N$ be a subset of its vertices, such that for every two vertices in $N$ no one is a descendant of the other.
We say that a vertex $u$ of $T$ is {\em $N$-critical} if either it
belongs to $N$ or there are at least two vertices
in $N$ that are descendants of two distinct children of $u$.  An  {\em $N$-unimportant} path in $T$ is a path with at least $2$ vertices,
with exactly two $N$-critical vertices, which 
are its endpoints (see \cref{fig:unim} for a picture).  Notice that an
$N$-unimportant path in $T$ cannot have an internal vertex that
belongs to some other $N$-unimportant path. Also, among the two
endpoints of an $N$-unimportant path there is always one which is a
descendant of the other. As we see in the proof of the following lemma, $N$-unimportant paths are the
maximal paths with internal vertices of degree 2 that appear if we repeatedly delete
leaves that do not belong to~$N$.

\begin{figure}[ht]
  \centering
  \begin{tikzpicture}[every node/.style = black node]
    \draw (0,0) node[label=90:root] (top) {} (2,-1) node (u) {} (-2,-1) node[normal] (left) {} (4,-2) node[normal] (right) {};
    \draw (right) -- ++(-60:1) -- +(180:1) -- (right) -- (u) -- (top) -- (left) -- ++(-60:1) -- +(180:1) -- (left);
    \draw[dashed] (u) -- ++(-2,-1) node[solid] (mid1) {} -- ++(-2,-1) node[solid] (mid2) {}
    -- ++(-2,-1) node[solid] (v) {};
    \draw (v) -- ++(-60:1) -- +(180:1) -- (v);
    \draw (mid1) -- ++(2,-1) node[normal] (mid1l) {};
    \draw (mid2) -- ++(2,-1) node[normal] (mid2l) {};
    \draw (mid2) -- ++(0,-1) node[normal] (mid2d) {};
    -- (mid2d);
    \draw[fill = gray!50] (mid1l) -- ++(-60:1) -- +(180:1) -- (mid1l);
    \draw[fill = gray!50] (mid2l) -- ++(-60:1) -- +(180:1) -- (mid2l);
    \draw[fill = gray!50] (mid2d) -- ++(-60:1) -- +(180:1) -- (mid2d);
  \end{tikzpicture}
  \caption{An unimportant path (dashed) in a tree. Gray subtrees are those
    without vertices from~$N$.}
  \label{fig:unim}
\end{figure}
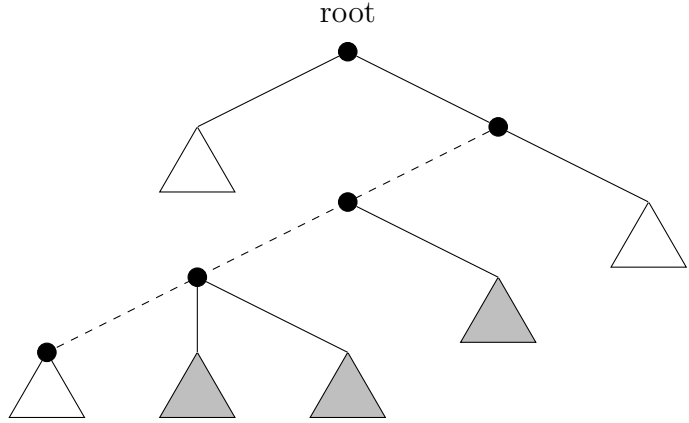


\begin{lemma}
  \label{wt5lokp}
  Let $d,k \in \N$, $k \geq 1$. 
  Let $(T,s)$ be a rooted tree and let $N$ be a set of vertices of 
  $(T,s)$,  each of which is at distance at least $d$ from $s$.
  If for some integer $k$, every $N$-unimportant path in $T$ has length at most $k$, then $|N|\geq 2^{d/k}$.
\end{lemma}
 
\begin{proof}
We consider the subtree $T'$ of $T$ obtained by repeatedly deleting
leaves that do not belong to $N$.
For the purposes of this construction, the root of $T$ is considered a leaf if it has degree $1$ and, when deleted, the root of the new tree is the previous root's (unique) child. Let $s'$ denote the root of $T'$.
By construction, every leaf of $(T',s')$ belongs to $N$, hence our goal is then to show
that $(T',s')$ has many leaves.
Notice that in $(T',s')$, every vertex of degree at least 3 is
$N$-critical.
Therefore, the $N$-unimportant paths of $(T',s')$ are the maximal
paths, the internal vertices of which have degree two.
By contracting each of these paths into an edge, we obtain a tree
$T''$ where every internal vertex has degree at least 3.
Observe that every edge on a root-leaf path of $T''$ is originated
from the contraction of a path on at most $k$ edges, as we assume that
every $N$-unimportant path in $T$ has length at most $k$. We deduce that
$T''$ has height at least $d/k$, hence it has at least $2^{d/k}$
leaves. Consequently, $T'$ has at least $2^{d/k}$
leaves, and then $|N| \geq 2^{d/k}$.\qedhere

\end{proof}

 
Recall that if $({\cal X},T,s)$ is a distance-decomposition of a
graph and $t\in V(T)\setminus\{s\}$, $E^{(t)}$ denotes as the set of edges that have
one endpoint in $X_{t}$ and the other in $X_{{\bf p}(t)}$.

\begin{lemma} \label{blurp}
Let $G$ be an $n$-vertex graph, let $r$ be a positive integer, let
${\cal D}=({\cal X},T,s)$ be a distance-decomposition of $G$, and let $d>1$ be the height of  
$(T,s)$. Then either $G$ contains a $\theta_{r}$-model with at most $2\cdot {r} \cdot {d}$ edges or 
for every vertex $i\in V(T)\setminus{s}$, it holds that $|E^{(i)}|\leq r-1$.
Moreover there exists an algorithm that, in $O_r(m)$ steps, either
finds such a model, or asserts that $|E^{(i)}|\leq r-1$ for every $i\in V(T)\setminus{s}$.
\end{lemma}

 \begin{proof} We consider the non-trivial case where $r\geq 2$.
Suppose that there exists a node $t$ of $(T,s)$ such that  $|E^{(t)}|\geq r$.
Clearly, such a $t$ can be found in $O(m)$ steps.
We will prove that $G$ contains a $\theta_{r}$-model. Let $k$ be the height of $t$ in $T$.

We need first the following claim.
\begin{claim}\label{c:1}
Given a non-empty proper subset $U$ of $X_{t}$, we can find in $G_t$ a path of length at most $2 k$
from a vertex of $U$ to a vertex of $X_t\setminus U$, in $O(m)$~steps.
\end{claim}

\begin{proofclaim}{of~\cref{c:1}}
We can compute a shortest path $P$ from a vertex of $U$ to a vertex of
$X_t\setminus U$, in $O(m)$~steps using a BFS.
 Let us show that $P$ has length at most $2k$. Let $u\in U$ and $v \in
 X_t\setminus U$ be the endpoints of $P$,
and let $w$ be a
vertex of $P$ of minimum height $h$ ($0\leq h \leq k$).
Then it holds that ${\bf dist}_{G_{t}}(v,u)={\bf dist}_{G_{t}}(U,v)$.
We examine the non-trivial case where $P$ has more than one edge.
By minimality of $P$ we have $w\notin X_{t}$.

Our next step is to prove that if $P$ has more than one edge,
then both the subpaths of $P$ from $u$ to $w$ and from $v$ to $w$
are straight.  Suppose now, without loss of generality, that the
subpath from $u$ to $w$ is not straight and let $z$ be the first vertex of it
(starting from $u$) which is contained in a bag of height greater than
or equal to the height of the bag of its predecessor in $P$. 
By definition of a distance-decomposition (in particular items \itemref{it:dist}
and \itemref{it:conn}), there is at least one vertex $x\in
X_{t}$ which is connected by a straight path $P'$ to $z$ in $G$. Then there
are two possibilities:

\begin{itemize}
\item either $x\in U$, and then the union of the path $P'$ and the portion of
$P$ between $z$ and $v$ is a path that is shorter than $P$;
\item or $x\in X_{t}\setminus U$, and in this case the union of the path $P'$ and
the portion of $P$ between $u$ and $z$ is a path that is shorter than
$P$.
\end{itemize} As, in both cases, the occurring paths contradict the
construction of $P$, we conclude that both the subpath of $P$ from
$u$ to $w$ and the one from $v$ to $w$ are straight.  This implies that $P$ has length at most $2\cdot(k-h)\leq
2\cdot k$ and the claim follows.
\end{proofclaim}

Our next step is  to construct a vertex set $U$ and a set of paths
${\cal P}$ as follows. We set ${\cal P}=\emptyset$, $U=\emptyset$, and
we start by adding in $U$ an arbitrarily chosen vertex $u\in X_{t}$. 
Using the procedure of \cref{c:1}, we repeatedly find a path from
a vertex of $U$ to a vertex of $X_t \setminus U$, add this second
vertex to $U$ and the path to $\cal P$, until there are at least $r$
edges in $E^{(t)}$ that have endpoints in $U$.

The construction of $U$ requires at most $r$ repetitions of the
procedure of \cref{c:1}, and therefore $O(r\cdot m)$ steps in
total. Clearly $|U|\leq r$, hence $|{\cal P}| \leq r-1$. Besides, every
path in $\cal P$ has length at most $2k$ according to \cref{c:1}.
Notice now that $\bigcup{\cal P}$ is a connected subgraph of $G_{t}$
with at most $2k \cdot(r-1)$ edges.

As there are at least $r$ edges in $E^{(t)}$ with endpoints in $U$ we
may consider a subset $F$ of them where $|F|=r$. Since $\cal D$ is a
distance-decomposition (by item \itemref{it:dist} of the
definition), each edge $e\in F$ is connected to the origin by a path of length $d-k-1$ whose edges do not belong
to~$G_{t}$. Let ${\cal P'}$ be the collection of these paths. Clearly,
the paths in ${\cal P'}$ contain, in total, at most $r\cdot (d-k-1)$~edges.

If we now contract in $G$ all edges in ${\cal P}$ and all edges in ${\cal P'}$, except those in $F$,
and then remove all edges not in $F$, we obtain a graph isomorphic to 
$\theta_{r}$. Therefore we found in $G$ a $\theta_r$-model with at most 

\begin{flalign*}
&& r\cdot(d-k-1)+2\cdot k \cdot(r-1) +r
&\leq r\cdot(d-k-1)+2\cdot k \cdot r +r &&\\
&& &= r\cdot(d+k)  \\
&& &\leq 2\cdot r\cdot d && \text{(since $ d\geq k$)}
\end{flalign*}
edges in $O(r\cdot m)$~steps.
\end{proof}

The following result is a direct consequence of \cref{blurp} and
item \itemref{it:dist} of the definition of a distance-decomposition.
\begin{corollary}
\label{corfkior}
Let $G$ be an $n$-vertex graph, let $r$ be a positive integer, let
${\cal D}=({\cal X},T,s)$ be a distance-decomposition of $G$, and let $d>1$ be the height of  $(T,s)$.
If some bag of ${\cal D}$ contains at least ${r}$ vertices, then $G$
contains a $\theta_{r}$-model  with at most $2\cdot {r} \cdot {d}$ edges, which can be found in $O_r(m)$~steps.\qed
\end{corollary}

The remaining lemmata are related to \hyperref[par:group]{grouped partitions}.
 \begin{lemma}
 \label{k5hfyeu}
For every positive integer $d$ and every  connected graph $G$ there is a $d$-grouped
partition of $G$ that can be constructed  in $O(m)$ steps.
  \end{lemma}
 
 \begin{proof} 
 If ${\bf diam}(G) \leq 2d$, then $\{V(G) \}$ is a $d$-grouped partition of~$G$. 
Otherwise, let $R=\{s_1,\dots, s_l\}$ be a maximal $2d$-scattered set
in $G$. This set can be constructed in $O(m)$~steps by~breadth-first search.
The sets $\{R_{i}\}_{i \in \intv{1}{l}}$ are constructed by the following procedure:
\begin{enumerate}
\item Set $k=0$ and $R_i^0 = \{s_i\}$ for every $i \in \intv{1}{l}$;
\item \label{e:add}For every $i \in \intv{1}{l}$, every $v \in R_{i}^{k}$ and every
  $u \in N_{G}(v)$, if $u$ has not been considered so far,
add $u$ to  $R_{i}^{k+1}$;
\item If $k< 2d$, increment $k$ by 1 and go to step~\ref{e:add};
\item Let $R_i = {\bigcup}_{k=0}^{2 d} R_{i}^{k}$ for every $i \in \intv{1}{l}$.
\end{enumerate}
Let ${\cal R} = \{R_{i}\}_{i \in \intv{1}{l}}$. By construction, each
set $R_{i}$ induces a connected graph in $G$. It remains to prove that
${\cal R}$ is a partition of $V(G)$ and that it has the desired properties.

Notice that in the above construction if a vertex is assigned to the
set $R_{i}$,
then it is not assigned to $R_{j}$, for every distinct integers $i,j\in \intv{1}{l}$. 
Let $v\in V(G)$ be a vertex that does not belong to $R_{i}$ for any $i\in
\intv{1}{l}$ after the procedure is completed.
Then for every $i \in \intv{1}{l}$ we have ${\bf dist}_{G}(v, s_i)>2d$
and $v \notin R$, which contradicts the maximality of~$R$.
Therefore ${\cal R}$ is a partition of $V(G)$.

Since for each vertex $v$ in $R_i$ it holds that ${\bf dist}_{G}(v, s_i)\leq 2d$, ${\cal R}$ obviously satisfies property~\itemref{it:gp1} of the definition.

For property~\itemref{it:gp2} of the definition, let $e=\{x,y\}$ be an
edge in $G$ such that $x\in R_i$, $y\in R_j$, for some distinct
integers $i,j \in \intv{1}{l}$. Towards a contradiction, we assume
without loss of generality that ${\bf dist}_{G}(x, s_i)<d$.  This
means that during the construction of $R_i$, the vertex $x$ was added
to the set $R_{i}^{k}$ for some $k\leq  d-1$.  Also, since the vertex $y$
is adjacent to $x$ but was added to $R_{j}^{l}$ for some $l\leq 2d$
instead of $R_{i}^{k+1}$, it follows that $l\leq k+1$, which means
that ${\bf dist}_{G}(y, s_j)\leq k+1$.  Hence ${\bf dist}_{G}(s_i,
s_j)\leq {\bf dist}_{G}(s_i, x) +{\bf dist}_{G}(x,y) +{\bf
dist}_{G}(y, s_j) \leq k+1+k+1\leq 2d$ again is not possible
since $R$ is a $2d$-scattered~set.

Finally, in the procedure above, each edge of the graph is encountered
at most once, hence the whole algorithm will take at most $O(m)$
time. This concludes the proof of the lemma.
\end{proof}
 
\begin{lemma}
\label{spoof}
Let $G$ be a graph, let ${\cal R}=\{ R_1, \ldots, R_l\}$ be a
$d$-grouped partition of $G$, and let $s_{i}$ be a center of
$R_{i}$, for every $i\in \intv{1}{l}$.
If for some distinct $i,j \in \intv{1}{l}$, $G$ has at least $r$ edges
from vertices in $R_i$ to vertices in $R_j$ then $G[R_{i}\cup R_{j}]$
contains a $\theta_{r}$-model   with at most $4\cdot {r} \cdot {d} +
r$ edges, which can be found in $O_{r}(m)$ steps.
\end{lemma}
\begin{proof} 
Suppose that for some $i\in \intv{1}{l}$, $G$ has a set $F$ of at least $r$
edges from vertices in $R_i$ to vertices in $R_j$. Let $R'_i\subseteq
R_i$ and $R'_j\subseteq R_j$ be the sets of the endpoints of those
edges. Since ${\cal R}$ is a $d$-grouped partition of $G$, it holds
that, for each $x\in R'_i$ and $y\in R'_j$, ${\bf dist}_{G}(x, s_i)
\leq 2d$ and ${\bf dist}_{G}(y, s_j) \leq 2d$. That directly implies
that for every $h\in \{i,j\},$ there is a collection ${\cal P}_{h}$ of
$r$ paths, each of length at most $2d$ and not necessarily disjoint,
in $G[R_h]$ connecting $s_h$ with each vertex in $R'_h$, which we can
find in $O_r (m)$ steps. It is now easy to observe that the graph $Q$,
obtained from $\bigcup{{\cal P}_{i}}\cup \, \bigcup{{\cal P}_{j}}$ by
adding all edges of $F$, is the union of $r$ paths between $s_{i}$ and
$s_{j}$, each containing at most $4\cdot d+1$ edges. Therefore, $Q$ is a model
of $\theta_{r}$ with at most  $4\cdot {r} \cdot {d} +r$ edges, as
required. As mentioned earlier the construction of ${\cal P}_{i}$ and
${\cal P}_{j}$ takes $O_{r}(m)$~steps.
\end{proof}
 
\begin{lemma}
\label{l:pong}
Let $G$ be a graph, let ${\cal R}=\{ R_1, \ldots, R_l\}$ be a
$d$-grouped partition of $G$, and let $S=\{s_{1},\ldots,s_{l}\}$ be a set
of centers of ${\cal R}$. 
For every $i\in \intv{1}{l}$, let ${\cal D}_{i}=({\cal X}_{i}, T_{i},
r_{i})$ be the distance-decomposition with origin $s_i$ of the graph~$G[R_{i}]$.
Let also $r$ be a positive integer such that for every $i\in \{1,\dots,l \}$ and for every vertex $t \in V(T_{i}) \setminus \{r_{i}\}$ it holds that $| E^{(t)}| \leq r-1$.
 If for some $i\in \intv{1}{l}$ and $w \in \N$, the tree $T_{i}$, with node-frontier $N_{i}$, has
 an $N_{i}$-unimportant path of length at least $2(w+1)$, then $G$ 
 has a connected $(2r-2)$-edge-protrusion $Y$ with extension more
 than~$w$, which can be constructed in $O_r(m)$ steps.
\end{lemma}

\begin{proof}
  Let $P = t_0\dots t_p$ be an
$N_{i}$-unimportant path of length $p\geq 2(w+1)$ in
$T_{i}$. We assume without loss of generality that $t_p \in
\textbf{des}_{(T_{i},r_{i})}(t_0)$.  Due to the definition of
distance-decompositions, the vertices in $X_{t_0}^{i}$ or $X_{t_p}^{i}$
form a vertex-separator of $G$. Let $Z\subseteq E(G)$ be the set
containing all edges between $X_{t_0}^{i}$ and $X_{t_1}^{i}$ and all
edges between $X_{t_{p-1}}^{i}$ and $X_{t_p}^{i}$ in~$G$. Clearly, $Z$ is an
edge-separator of $G$ with at most $2r-2$ edges. Let $T'_{i}$ be the
subtree of $T_{i}$ that we obtain if we remove the descendants of $t_p$ that are distinct from this vertex
and any vertex that is not $t_{0}$ or a descendant of $t_1$. Let $Y =
\bigcup_{t\in V(T'_{i})\setminus\{t_0,t_p\}} X_{t}^{i}$.  In other words,
$Y$ consists of the vertices in the bags of $T'_{i}$ excluding
$X_{t_{0}}^{i}$ and $X_{t_{p}}^{i}$. Obviously, $N_{G}(Y) = X_{t_0}\cup X_{t_p}$.

We will now construct a rooted tree-partition ${\cal F}=({\cal X_{\cal
F}},T_{\cal F},r_{\cal F})$ of $G[Y \cup N_{G}(Y)]$ of width at most
$2r-2$ and such that $|V(T_{\cal F})|>w$. Let $T_{\cal F}$ be the tree obtained from $T'_{i}$ by identifying, for every $j \in \intv{0}{\lfloor
(p-1)/2\rfloor}$, the vertex $t_{j}$ with the vertex $t_{p-j}$. If
multiple edges are created during this identification, we replace them
with simple ones. We also delete loops that may be created. Let us define the elements of ${\cal X^{\cal F}}=\{X^{\cal F}_{t}\}_{t\in
V(T_{F})}$ as follows. If $t\in V(T_F)$ is the result of the identification of
$t_{j}$ and $t_{p-j}$ for some $j \in \intv{0}{\lfloor
(p-1)/2\rfloor}$, then we set  $X_{t}^{\cal F} = X_{t_{j}}\cup
X_{t_{p-j}}$. On the other hand, if $t\in V(T_F)$ is a vertex of
$T_{i}'$ that has not been identified with some other vertex, then
$X_{t}^{\cal F}=X_{t}$.  The construction of ${\cal F}$ is completed by setting
$r_{\cal F}$ to be the result of the identification of $t_0$ and
$t_p$, the endpoints of~$P$.

It is easy to verify that ${\cal F}$ is a rooted tree-partition of
$G[Y \cup N_{G}(Y)]$ of width at most $2r-2$.  Notice also that the
identification of the antipodal vertices of the path $P$
creates a path in $T_{\cal F}$ of length $\lfloor (p-1)/2\rfloor$.  This
implies that the extension of ${\cal F}$ is at least $\lfloor
(p-1)/2\rfloor\geq w+1$. Besides, all the operations performed to
construct $\cal F$ can be implemented in $O_{r}(m)$ steps. This completes the proof.
\end{proof}

We conclude this section with two easy lemmata related
to~\hyperref[par:ports]{ports and frontiers}.
\begin{lemma}
\label{fnorks}
Let $G$ be a graph, let ${\cal R}=\{ R_1, \ldots, R_l\}$ be a
$d$-grouped partition of $G$, and let $S=\{s_{1},\ldots,s_{l}\}$ be a
set of centers of ${\cal R}$. For every $i \in \intv{1}{l}$, let ${\cal D}_{i}=({\cal X}_{i}, T_{i},
r_{i})$ be the distance-decomposition with origin~$s_i$ of the
graph~$G[R_{i}]$, and let $N_{i}$ be the node-frontier of $T_i$.
Then, for every $i\in\intv{1}{l}$, there are at least $|N_{i}|$ ports in~$T_{i}$.
\end{lemma}

\begin{proof}
Let $i \in \intv{1}{l}$. We will show that every vertex in the node-frontier of $T_{i}$ has a descendant which is a port.
For every vertex $t \in N_{i} \subseteq V(T_{i})$, there is, by
definition, a path from $t$ to a vertex in $G \setminus R_i$, the
internal vertices of which belong to $V_i^{\geq d}$.
Let $v$ be the last vertex of this path (starting from $t$) which
belongs to $R_i$ and let $t'\in V(T)$ be the vertex such that $v \in
X^i_t$. Then $t'$ is a port of~$T_i$.
Observe that $t'$ cannot be the descendant of any other vertex of
$N_i$. Therefore there are at least $|N_i|$ ports in~$T_i$.
\end{proof}
 
\begin{corollary}
\label{ping}
Let $G$ be a graph, let ${\cal R}=\{ R_1, \ldots, R_l\}$ be a
$d$-grouped partition of $G$, and let $S=\{s_{1},\ldots,s_{l}\}$ be a set
of centers of ${\cal R}$. For every $i\in \intv{1}{l}$, let ${\cal D}_{i}=({\cal X}_{i}, T_{i},
r_{i})$ be the distance-decomposition with origin~$s_i$ of the
graph~$G[R_{i}]$, and let $N_{i}$ be the node-frontier of~$T_i$. If for some integer $k$, every
$N_{i}$-unimportant path in $T_{i}$ has length at most $k$, then
$T_{i}$ contains at least $2^{d/k}$ ports.\qed
\end{corollary}

\begin{proof}
Let $i\in \intv{1}{l}$. From \cref{fnorks}, it is enough to prove that 
$|N_{i}|\geq 2^{d/k}$. Then the result follows by applying 
\cref{wt5lokp} for $(T_{i}, s_{i})$, $d$, $N_{i}$, and~$k$.
\end{proof}

\subsection{Proof of \texorpdfstring{\cref{rk6ter}}{the second intermediate theorem}}
\label{sec:proofth5}
 
\begin{proof}
Let $d=\frac{z-r}{4r}$. According to \cref{k5hfyeu}, we can
construct in $O(m)$ steps a $d$-grouped partition ${\cal R}=\{ R_1,
\ldots, R_l\}$ of $V(G)$, with a set of centers
$S=\{s_{1},\ldots,s_{l}\}$, and also, for every $i\in\intv{1}{l}$, the
distance-decompositions ${\cal D}_{i}=({\cal X}_{i},
T_{i}, r_{i})$ with origins $s_{i}$ of the graphs~$G[R_{i}]$. For every $i \in \intv{1}{l}$, we use the
notation ${\cal X}_{i}=\{X^{i}_{t}\}_{t\in V(T_{i})}$ and denote by
$N_{i}$ the node-frontiers of~$T_{i}$.

By applying the algorithm of \cref{spoof}, in $O_{r}(m)$ steps, we
either find a $\theta_{r}$-model in $G$ with at most $z=4\cdot r\cdot
d+r$ edges or we know that for every two distinct $i,j\in \intv{1}{l}$ there
are at most $r-1$ edges of $G$ with one endpoint in $R_i$ and one in
$R_j$.

Similarly, by applying the algorithm of \cref{blurp}, in $O_r(m)$ steps we
either find a $\theta_r$-model in $G$ with at most $2\cdot r\cdot d\leq z$ edges (in which case we are done) or we know
that for every $i\in\intv{1}{l}$ and every $t\in V(T_i)\setminus{s_i}$, it holds that $|E^{(i)}|\leq r-1$

In the second case, by using the algorithm of \cref{l:pong}, in $O_r(m)$ steps we either
find a connected $(2r-2)$-edge-protrusion with extension more than $w$, or we know
that for every $i\in \intv{1}{l}$, all $N_{i}$-unimportant paths of $T_{i}$ have length at most~$2w+1$.

We may now assume that none of the above algorithms provided a
$\theta_r$-model with $z$ edges, or a $(2r-2)$-edge-protrusion.

From \cref{ping}, for every $i \in \intv{1}{l}$ the tree $T_i$
contains at least $2^{\frac{d-1}{2w+1}}=2^{\frac{z-5r}{4r\cdot (2w+1)}}$
ports, which by definition means that
there are at least $2^{\frac{z-5r}{4r\cdot (2w+1)}}$ edges in $G$ with one endpoint in $R_{i}$ 
and the other in $V(G)\setminus R_{i}$.
By \cref{spoof}, for every distinct integers $i,j\in \intv{1}{l}$ there are at
most $r-1$ edges with one endpoint in $R_{i}$ and the other in~$R_{j}$.
As a consequence of the two previous implications, for every $i\in \intv{1}{l}$ there
is a set $Z_{i}\subseteq \intv{1}{l}\setminus\{i\}$, where $|Z_{i}|\geq
\frac{1}{r-1}2^{\frac{z-5r}{4r(2w+1)}}$,
such that for every $j\in Z_{i}$ there exists an edge with one endpoint in $R_{i}$ and the other in $R_{j}$.
Consequently, if we now contract all edges in $G[R_{i}]$ for every $i \in \intv{1}{l}$, the resulting graph $H$ 
is a minor of $G$ of minimum degree at least
$\frac{1}{r-1}2^{\frac{z-5r}{4r(2w+1)}}$. Therefore, we output $G$,
which is an $H$-model, as required in this case.
\end{proof}

\subsection{Proof of \texorpdfstring{\cref{sl0tr}}{the first intermediate theorem}}
\label{sec:proofth4}

\begin{proof}
The proof is quite similar to the one of \cref{rk6ter}.
If $G$ contains a vertex $v$ of degree less than $\delta$, we can easily find it in $O(m)$ steps. Hence, from now on we can assume that every vertex has degree at least $\delta$.

Let $d=\frac{z-r}{4r}$. From \cref{k5hfyeu}, in $O(m)$ steps, we can construct a $d$-grouped 
partition  ${\cal R}=\{ R_1, \dots, R_l\}$ of $G$, with a 
set of centers $S=\{s_{1},\ldots,s_{l}\}$, and also the
distance-decomposition ${\cal D}_{i}=({\cal X}_{i}, T_{i}, r_{i})$ with origins $s_i$ of
the graphs $G[R_{i}]$, for every~$i\in \intv{1}{l}$.
We use again the notation ${\cal  X}_{i}=\{X^{i}_{t}\}_{t\in V(T_{i})}$.

As in the proof of \cref{rk6ter}, in $O_{r}(m)$ steps, we can either
find a $\theta_{r}$-model in $G$ with at most $z=4\cdot r\cdot
d+r$ edges or we know that for every distinct integers $i,j\in[l]$ there are at most 
$r-1$ edges of $G$ with one endpoint in $R_{i}$ and one in~$R_{j}$
(\textit{cf.} \cref{spoof}).

Using \cref{corfkior}, we can in $O_{r}(m)$ steps either find a
$\theta_r$-model in $G$ with at most $z$ edges or we know that every bag
of ${\cal D}_{i}$ has less than $r$ vertices, for every~$i \in  \intv{1}{l}$.
Let $i \in \intv{1}{l}$ and let $u \in R_i$ be a vertex at distance less
than $d$ from $s_i$.
As $u$ has degree at least $3r$, it must have
neighbors in at least $3$ different bags of ${\cal D}_i$, apart from
the one containing it.
This means that every vertex in $T_{i}$ of distance less than $d$ from $r_{i}$ has degree at least $\lfloor \frac{\delta}{r-1} \rfloor \geq 3$ and
therefore $T_{i}$ has at  least ${\lfloor \frac{\delta}{r-1}-1 \rfloor }^{d}$ leaves.
Notice also that if $t$ is a leaf of $T_{i}$, then each vertex in $X_{t}^{i}$ can have at most $r-1$ neighbors 
in $X_{{\bf p}(t)}^{i}$  and at most $r-2$ neighbors in $X_{t}^{i}$.
Therefore there are at least $\delta-(r-1)-(r-2)=\delta-2r+3$ edges in $G$ with 
one endpoint in $X_{t}^{i}$ and the other in $V(G)\setminus R_{i}$.
This means that for every $i\in \intv{1}{l}$ there are at least
$(\delta-2r+3)\cdot {\lfloor \frac{\delta}{r-1}-1 \rfloor }^{d}$ edges
with one endpoint in $R_{i}$ and the other $V(G)\setminus R_{i}$.

Similarly to the proof of \cref{rk6ter}, we deduce that, 
for each $i\in\intv{1}{l}$, there is a set $Z_{i}\subseteq \intv{1}{l}\setminus\{i\}$ where $|Z_{i}|\geq \frac{\delta-2r+3}{r-1}\cdot {\lfloor \frac{\delta}{r-1}-1 \rfloor }^{d}$
such that, for every $j\in Z_{i}$, there exists an edge with one endpoint in $R_{i}$ and the other in $R_{j}$. This 
implies the existence of an $H$-model in $G$ for some $H$ with
$\delta(H)\geq \frac{\delta-2r+3}{r-1}\cdot {\lfloor
  \frac{\delta}{r-1}-1 \rfloor }^{\frac{z-r}{4r}}$. We then output
$G$, which, in this case, is an $H$-model.
\end{proof}

 \section{Excluding \texorpdfstring{$k$}{k} copies of \texorpdfstring{$\theta_{r}$}{a pumpkin} as a minor}
\label{sec:last}

This section is devoted to the proof of the following theorem.

\begin{theorem}
\label{s:new}
For every  graph $G$, $r\geq 2$, and 
$k\geq 1$, if $\tw(G) \geq 2^{6r}\cdot k\cdot  \log (k+1)$, then $G$ contains  $k\cdot \theta_{r}$ as a minor.
\end{theorem}

For the proof, we need to introduce some definitions and related results.

\subsection{Preliminaries}
\label{sec:def2}

Let $G$ be a graph and $G_1,G_2$ two non-empty subgraphs of~$G$.
We say that $(G_1, G_2)$ is a \emph{separation} of $G$ if:
\begin{itemize}
\item $V(G_1)\cup V(G_2) = V(G)$; and
\item $(E(G_1), E(G_2))$ is a partition of $E(G)$.
\end{itemize}

Let $G$ be a graph. Given a set $E\subseteq E(G)$, we define $V_{E}$ as the set of all endpoints of the edges in $E$.
Given a partition $(E_{1},E_{2})$ of $E(G)$ we 
define $\delta(E_1, E_2)=|V_{E_1}\cap V_{E_2}|$.

A \emph{cut} $C=(X,Y)$ of $G$ is a partition of $V(G)$ into two subsets $X$ and $Y$.
We define the \emph{cut-set} of $C$ as $E_C = \{ \{x,y\} \in E(G)\ | \ 
x\in X\ \text{and}\ y \in Y\}$ and call $|E_C|$ the \emph{order} of the cut.
Also, given a graph $G$, we denote by $\sigma(G)$ the number of connected components of $G$.

\paragraph{The branchwidth of a graph.}
  A \emph{branch-decomposition} of a graph $G$ is a pair $(T,\tau)$ where $T$ is a ternary tree and $\tau$ a bijection from the edges of $G$ to the leaves of $T$.
Deleting any edge $e$ of $T$ partitions the leaves of $T$ into two sets, and thus the edges of $G$ into two subsets $E_1^e$ and $E_2^e$.
 The \emph{width} of a branch-decomposition $(T,\tau)$ is equal to~$\max_{e\in E(T)}\{ \delta(E_1^e,E_2^e) \}$. 
 The \emph{branchwidth} of a graph $G$, denoted $\bw(G)$, is defined as the minimum width over all branch-decompositions of $G$.

\paragraph{The branchwidth of a matroid.}
We assume that the reader is familiar with the basic notions of matroid theory.
We will use the standard notation from Oxley's book~\cite{Ox92}.
The branchwidth of a matroid is defined very similarly to that of a graph.
Let ${\cal M}$ be a matroid with finite ground set $E({\cal M})$ and rank function $r$.
The \emph{order} of a non-trivial partition $(E_1,E_2)$ of $E({\cal M})$ is defined as $\lambda(E_1,E_2) =r(E_1) +r(E_2)-r(E)+1$.
A \emph{branch-decomposition} of a matroid ${\cal M}$ is a pair $(T,\mu)$ where $T$ is a ternary tree and $\mu$ is a bijection from the elements of $E({\cal M})$ to the leaves of $T$.
Deleting any edge $e$ of $T$ partitions the leaves of $T$ into two sets, and thus the elements of $E({\cal M})$ into two subsets $E_1^e$ and $E_2^e$.
The \emph{width} of a branch-decomposition $(T,\mu)$ is equal to~$\max_{e\in E(T)}\{ \lambda(E_1^e,E_2^e) \}$.
 The \emph{branchwidth} of a matroid ${\cal M}$, denoted $\bw({\cal M})$, is again defined as the minimum width over all branch-decompositions of ${\cal M}$.
The \emph{cycle matroid of a graph $G$} denoted ${\cal M}_G$, has ground set $E({\cal M}_G)=E(G)$ and the cycles of $G$ as the circuits of ${\cal M}_G$.
Let $G$ be a graph, ${\cal M}_G$ its cycle matroid and $(G_1,G_2)$ a separation of $G$. 
Then clearly $(E(G_1),E(G_2))$ is a partition of $E({\cal M}_G)$, but to avoid confusion we will henceforth denote it $(E_1,E_2)$
 and we will call it \emph{the partition of ${\cal M}_G$ that corresponds to the separation $(G_1,G_2)$ of $G$}.
Observe that the order of this partition is: 
 \begin{equation}\label{matr}
    \lambda(E_1,E_2) = \delta(E(G_1),E(G_2)) - \sigma(G_1) - \sigma(G_2) + \sigma(G) +1.\tag{$\star$}
  \end{equation}

\paragraph{Minor obstructions.} Let $\cal{G}$ be a graph class.
We denote by ${\bf obs}({\cal G})$  the set of all minor-minimal
graphs $H$ such that $H\notin \cal{G}$ and we will call it \emph{the
  minor obstruction set for $\cal{G}$}. Clearly, if $\cal{G}$ is
closed under minors, the minor obstruction set for $\cal{G}$ provides
a complete characterization for $\cal{G}$: a graph $G$ belongs in
$\cal{G}$ if and only if none of the graphs in ${\bf obs}({\cal G})$
is a minor of~$G$.

Given a class of matroids ${\bf M}$,  \emph{the minor obstruction set
  for ${\bf M}$},  denoted by ${\bf obs}({\bf M})$, is defined very
similarly to its graph-counterpart: it is simply the set of all
minor-minimal matroids ${\cal M}$ such that ${\cal M}\notin {\bf M}$.

We will need the following results.
\begin{proposition}[\!\!\protect{\cite[Theorem 5.1]{RS91}}]
\label{grminx}
Let $G$ be a graph of branchwidth at least 2. Then,
$\bw(G)\leq \tw(G)+1\leq\lfloor\frac{3}{2}\bw(G)\rfloor.$
\end{proposition}

\begin{proposition}[\!\!\cite{BodlaenderLTT97}]
\label{griii6ytnx}
Let $r\in\N_{\geq 1}$ and let $G$ be a graph. If $\bw(G)\geq 2r+1$, then 
$G$ contains a $\theta_{r}$-model. 
\end{proposition}

\begin{proposition}[\!\!\protect{\cite[Theorem 4]{HM07}}]
\label{prop1}
Let $G$ be a graph that contains a cycle
and ${\cal M}_G$ be its cycle matroid. Then, $\bw(G)=\bw({\cal M}_G)$.
\end{proposition}

\begin{proposition}[\!\!\protect{\cite[Lemma 4.1]{GeelenGR03onth}}]
\label{prop2}
Let a matroid ${\cal M}$ be a minor obstruction for the class of matroids of branchwidth at most $k$ and 
let $g(n)=(6^{n-1}-1)/5$. Then, for every partition $(X,Y)$ of ${\cal M}$ with $\lambda(X,Y) \leq k$, either $|X|\leq g(\lambda(X,Y))$ or $|Y|\leq g(\lambda(X,Y))$.
\end{proposition}

The following observations are also crucial.
\begin{observation}
\label{y7ury5to0o}
Let ${\cal G}$ be a graph class that is closed under minors and let ${\cal M}_{\cal G}=\{{\cal M}_{G} \mid G\in{\cal G}\}$.
${\cal G}$ is minor closed if and only if ${\cal M}_{\cal G}$ is minor closed.
Moreover, for every $H\in{\bf obs}({\cal G})$ it holds that ${\cal M}_{H}\in{\bf obs}({\cal M}_{\cal G})$.
\end{observation}

The above observation is a direct consequence of the definition of 
matroid remo\-val/contraction, e.g., see Proposition 4.9 of~\cite{pitsoulis2014topics}.

\begin{lemma}
\label{d:dsfdf}
There is a constant $c\in \R_{\geq 2}$, such that for any integer $k\geq r\geq 2$, if $g(n)=(6^{n-1}-1)/5$, then $\frac{1}{r-1}2^\frac{c^r\log{k}-5r}{4r(2g(2r-2)+1)} \geq k(r+1)-1$. 
Moreover,  this holds for $c={6}^{3}$.
\end{lemma}

\begin{proof}
We only present here the main steps of the computation. Let $c=6^3$. If $r\geq 4$, then we have:
\begin{align*}
  c^r &\geq6^{3r-3}\\
   &\geq\frac{8r\cdot6^{2r-3}(r^2-1)\log k}{\log k}\\
   &\geq\frac{4r[1+2(6^{2r-3}-1)]\log [k(r^2-1)-(r-1)] +5r}{\log k}\\
  c^r\log k-5r &\geq 4r[1+\frac{2}{5}(6^{2r-3}-1)]\log [k(r^2-1)-(r-1)]\\
  2^{\frac{c^r\log k-5r}{4r(1+2g(2r-2))}} &\geq k(r^2-1)-(r-1)\\
  \frac{1}{r-1}2^{\frac{c^r\log k-5r}{4r(1+2g(2r-2))}} &\geq k(r+1)-1.
\end{align*}
Finally, for $c=6^3$ and either $r=2$ or $r=3$, one can easily check that the inequality holds for any $k\geq r$.

\end{proof}

\subsection{Graphs with large minimum degree}
\label{sec:gwlmd}

In this subsection we show that every graph of large minimum degree
contains $k\cdot\theta_r$ as minor. Our proof relies on the following
result.

\begin{proposition}[\!\!\protect{\cite[Corollary 3]{JGT:Stiebitz}}]\label{stieb}
  For every $k,r \in \N_{\geq 1}$, every graph $G$ with $\delta(G)\geq k(r
  + 1) - 1$ has a partition $(V_1, \dots, V_k)$ of
  its vertex set satisfying $\delta(G[V_i]) \geq
  r$ for every $i \in \intv{1}{k}$.
\end{proposition}

\begin{lemma}\label{easy}
For every integer $r \in \N_{\geq 1}$, every graph of minimum degree at least
$r$ contains a $\theta_r$-model.  
\end{lemma}

\begin{proof}
  Starting from any vertex $u$, we grow a maximal path $P$ in $G$ by iteratively
adding to $P$ a vertex that is adjacent to the previously added
vertex but does not belong to~$P$. Since $\delta(G) \geq r$, any
such path will have length at least $r+1$.
At the end, all the neighbors of the last vertex $v$ of $P$ belong to $P$
(otherwise $P$ could be extended). Since $v$ has degree at least $r$,
$v$ has at least $r$ neighbors in $P$. Therefore $P$ is a
$\theta_r$-model in $G$.
\end{proof}

\begin{corollary}\label{l:vpackex}
For every $k,r \in \N_{\geq 1}$, every graph $G$ with $\delta(G) \geq k(r
+ 1) - 1$ contains a $k\cdot\theta_r$-model.
\end{corollary}

\begin{proof}
  According to \cref{stieb}, $V(G)$ has a partition $(V_1, \dots,
  V_k)$ such that $\delta(G[V_i])$ $\geq
  r$ for every $i \in \intv{1}{k}$. Therefore, by \cref{easy}, for every  $i \in
  \intv{1}{k}$ the graph $G[V_i]$ has a $\theta_r$-model $M_i$. Clearly $M_1 \cup \dots \cup M_k$ is a
  $k\cdot\theta_r$-model in $G$, as~desired.
\end{proof}
Now we are ready to prove the main result of this section.
\subsection{Proof of \texorpdfstring{\cref{s:new}}{the main result of the section}}
\label{sec:pr}

For every $r \in \N$, we define $f(r)=\frac{2}{3}2^{6 r}$.
By \cref{grminx}, it is enough to prove that 
if $\bw(G)\geq f(r)\cdot k\cdot \log(k+1)$, then $G$ contains
$k \cdot \theta_r$ as a minor.
To prove this we use induction on~$k$.

The case where $k=1$ follows from \cref{griii6ytnx} and the fact that  $f(r)\geq  2r+1$.
We now examine the case where $k>1$, assuming that the proposition holds for smaller values of $k$.
As $\bw(G)\geq f(r)\cdot k\cdot \log(k+1)$, $G$ contains a minor
obstruction $H$ for the class of graphs of branchwidth at most~$f(r)\cdot k\cdot \log(k+1)-1$.
We will in fact prove that $H$ contains $k \cdot \theta_r$ as a minor.

\begin{claim}\label{c:last}
Any $(2r-2)$-edge-protrusion of $H$ has extension at most~$g(2r-2)$.  
\end{claim}

\begin{proofclaim}{of \cref{c:last}}
Let $C=(X,Y)$ be a cut in $H$ of order at most $2r-2$ and let $H_X$ be
the subgraph of $H$ with $V(H_X)=X\cup N_H(X)$ and let $E(H_X)=E(H[X])\cup
E_C$. Clearly the pair $(H_X, H[Y])$ is a separation of $H$. 
Let ${\cal M}_H$ be the cycle matroid of $H$ and $(E_X,E_Y)$ be the partition of ${\cal M}_H$ 
that corresponds to the aforementioned separation.
By \cref{prop1}, $\bw({\cal M}_H)=\bw(H)\geq f(r)\cdot k\cdot \log(k+1)$ (as $\bw(H)\geq 3$, $H$  
is not acyclic). Therefore, 
by \cref{y7ury5to0o}, ${\cal M}_H$ is a minor obstruction for the class 
of matroids of branchwidth $f(r)\cdot k\cdot \log(k+1)-1$. 
We set $\lambda = \lambda(E_X,E_Y)$. From \itemref{matr}, we have: 
\begin{flalign*}
\lambda&= r(E_X) + r(E_Y) - r({\cal M}_H)
+1\\
&= \delta(E(H_X), E(H[Y])) - \sigma(H_X) - \sigma(H[Y]) + \sigma(H) +1 \\ 
&\leq \delta(E(H_X), E(H[Y])) \\
&\leq |E_C|= 2r-2  \\
&\leq f(r)\cdot k\cdot \log(k+1)-1.
\end{flalign*}

 Thus, by \cref{prop2}, either $|E_X|\leq g(\lambda)$ or $|E_Y|\leq g(\lambda)$.
Since $g$ is non-decreasing, either $|E(H_X)|\leq g(2r-2)$ or $ |E(H[Y])|\leq g(2r-2)$.
 This directly implies that for any $(2r-2)$-edge-protrusion $Z$ of $H$, $H[Z\cup N_H(Z)]$ has at most $g(2r-2)$ edges. 
 Therefore $Z$'s extension is also at most $g(2r-2)$ and the claim follows.
\end{proofclaim}

Combining the above claim, \cref{d:dsfdf}, and \cref{rk6ter},
we infer that either $H$ contains a $\theta_r$-model with at most $f(r)\cdot\log{k}$ edges,
or it contains a minor with minimum degree at least $\frac{1}{r-1}\cdot2^\frac{f(r)\log{k}-5r}{4r(2g(2r-2)+1)} \geq k(r+1)-1$.
If the second case is true, then by \cref{l:vpackex}, $H$ contains
$k \cdot \theta_r$ as a minor. We immediately deduce that $G$ contains $k \cdot \theta_r$ as a minor, which proves the inductive step. We now consider the first case.
Let $M$ be a $\theta_r$-model with a minimum number of edges. Because we are in the first case, $|E(M)| \leq f(r)\cdot\log{k}$.
Observe that, because $\theta_r$ is 2-connected (as we assume $r \geq 2$), there is some 2-connected component of $M$ which contains $\theta_r$ as a minor. We deduce from the minimality of $M$ that $M$ is is 2-connected. Therefore, 
 $|V(M)|\leq |E(M)|\leq f(r)\cdot\log{k}$ and we can bound the
 treewidth of the graph $H'=H\setminus V(M)$ as follows:
 \begin{align*}
   \tw(H') &\geq \tw(H) -|V(M)| \\
&\geq f(r)\cdot k\cdot\log(k+1)-f(r)\cdot\log{k}\\
&\geq f(r)\cdot k\cdot\log{k}-f(r)\cdot\log{k}\\
&= f(r)\cdot(k-1)\cdot\log{k}.
 \end{align*}
 
Then, from the induction hypothesis, $H'$ contains a
$(k-1)\cdot \theta_r$-model $M'$ and obviously $M \cup M'$ is a
$k \cdot \theta_r$-model in $H$. Again, this implies that $G$ contains $k \cdot \theta_r$ as a minor, which concludes the induction step.\qed

\medskip

\cref{s:new} implies that for every fixed $r$, it holds that every graph excluding $k\cdot\theta_{r}$
as a minor has treewidth $O(k\cdot \log k)$. 
We conclude with a lemma indicating that this bound is tight up to the constants hidden in the $O$-notation.

\begin{lemma}
There is an increasing sequence of integers $(k_i)_{i\in \N}$ and an
infinite sequence of graphs $(G_i)_{i\in \N}$ such that 
$\tw(G_i) = \Omega(k_i \log k_i)$
and $G_i$ does not contain $k_i\cdot \theta_{r}$ as a minor, for every
$r\in \N_{\geq 2}$.
\end{lemma}

\begin{proof}
According to \cite[Theorem 5.13]{Morgenstern1994exis}, there is an
infinite family $\{G_i\}_{i\in \N}$ of 3-regular Ramanujan graphs
$G_i$ such that $i \mapsto |G_i|$ is an increasing function.
Furthermore, for every $i\in \N$, the graph $G_i$ has girth at least $\frac{2}{3} \log |V(G_i)|$ (\cite[Theorem 5.13]{Morgenstern1994exis})  and satisfies~$\tw(G_i)=\Omega(|V(G_i)|)$
(see~\cite[Corollary~1]{Bezrukov2004155}).
For every $i \in \N$, let $k_i$ be the minimum integer such that
$|V(G_i)| < k_i\cdot \frac{2}{3}\log |V(G_i)|$. Observe that
$(k_i)_{i\in \N}$ is increasing.
Notice  
that $|V(G_i)|=\Omega(k_i \cdot \log k_i )$, and thus $\tw(G_i)=\Omega(k_i\cdot  \log k_i)$.
We will show that $G_i$ does not contain $k_i$ vertex-disjoint cycles,
which implies that $k_i\cdot \theta_{r}$ is not a minor of $G_i$, for
every $r\in \N_{\geq 2}$.
Suppose for contradiction that $G_i$ contains 
$k_i$ vertex-disjoint cycles.
As the girth of $G_i$ is at least $\frac{2}{3} \log |V(G_i)|$,
each of these cycles has at least  $\frac{2}{3}\log |V(G_i)|$ vertices.
Therefore $G$ should contain at least 
 $k\cdot \frac{2}{3} \log |V(G_i)|$ vertices. This implies that
 $|V(G)|\geq k\cdot \frac{2}{3} \log |V(G_i)| > |V(G_i)|$, a
 contradiction.
Therefore $(k_i)_{i\in \N}$ and $(G_i)_{i\in \N}$ satisfy the required properties.
\end{proof}

\section{Concluding remarks}

In this paper, we introduced the concept of $H$-girth and proved that
for every $r \in \N_{\geq 2}$, a 
{\sl large} $\theta_r$-girth forces an exponentially large clique
minor. This extends the results of Kühn and Osthus related to the
usual notion of girth. We also gave a variant of our result where the
minimum degree is replaced by a connectivity measure.
As an application of our result, we optimally improved (up to a
constant factor) the upper-bound on the treewidth of graphs excluding
$k\cdot \theta_r$ as a minor.
A first question is whether our lower-bound on the clique
minor size can be improved.

Let us now state more general questions
spawned by this work.
A natural line of research is to investigate the $H$-girth
parameter for different instanciations of~$H$. An interesting problem in this
direction could be to characterize the graphs $H$ for which our results (\cref{t:klikforce} and \cref{t:klikforce2}) can be extended.

From its definition, the $H$-girth is related to the minor relation.
An other direction of research would be to extend the parameter of
$H$-girth to other containment relations.
One could consider, for a fixed graph $H$, the minimum
size of an induced subgraph that can be contracted to $H$, or the
minimum size of a subdivision of $H$ in a graph. The first one of
these parameters is related to induced minors and the
second one to topological minors.

As the usual notion of girth appears in various contexts in
graph theory, we wonder for which graphs $H$ the results related to
girth can be extended to the $H$-girth or to the two aforementioned variants.

\bibliographystyle{abbrv}

\end{document}